\documentclass[11pt,twoside]{amsart}
\usepackage{MnSymbol}
\usepackage{tikz-cd}
\usetikzlibrary{arrows.meta}
\tikzset{
  commutative diagrams/.cd,
  arrow style=tikz,
  diagrams={>={Computer Modern Rightarrow[length=5pt,width=3pt]}},
}
\usepackage{amsmath}
\usepackage{amsthm}
\usepackage{latexsym}
\usepackage{amscd}
\usepackage{ifthen}
\usepackage{graphicx}
\usepackage{enumerate}
\usepackage{soul}
\usepackage{mathrsfs}

\usepackage{hyperref}
\hypersetup{
    colorlinks = true,
    linkcolor = blue,
    filecolor = blue,      
    urlcolor = blue,
    citecolor = blue,
    breaklinks = true
}

\theoremstyle{plain}
\newtheorem{thm}{Theorem}[section]
\newtheorem{prop}[thm]{Proposition}

\newtheorem{lemma}[thm]{Lemma}

\theoremstyle{definition}
\newtheorem{defn}[thm]{Definition}

\newtheorem{rem}[thm]{Remark}

\newcommand {\Sing}{\mathrm{Sing}}
\newcommand{\Pic}{\operatorname{Pic}}
\newcommand{\Div}{\operatorname{Div}}

\newcommand{\Supp}{\operatorname{Supp}}
\newcommand{\h}{\operatorname{h}}
\newcommand{\Ho}{\operatorname{H}}
\newcommand{\ra}{\rightarrow}

\newcommand{\Hilb}{\operatorname{Hilb}}

\newcommand{\cliff}{\operatorname{Cliff}}
\newcommand{\syz}{\operatorname{Syz}}
\renewcommand{\O}{\mathcal{O}}
\newcommand{\I}{\mathcal{I}}
\renewcommand{\P}{\mathbb{P}}

\newcommand{\Spec}{\operatorname{Spec}}
\newcommand{\Sym}{\operatorname{Sym}}
\newcommand{\Hom}{\operatorname{Hom}}

\title[The bielliptic locus]{The bielliptic locus in the Hilbert scheme of canonical curves is unirational}

\author{Andrei Stoenică}

\address{A. Stoenicǎ: Faculty of Mathematics and Computer Science, University of Bucharest, Romania, and Simion Stoilow Institute of Mathematics of the Romanian Academy, Research group of the project "Cohomological Hall algebras of smooth surfaces and applications" - C.F. 44/14.11.2022} \email{andrei.stoenica@my.fmi.unibuc.ro}

\date{}


\begin{document}

\begin{abstract}
In this paper we prove the unirationality of the locus of bielliptic curves in the Hilbert scheme of canonical curves of genus $g \geq 11$. As a consequence, we obtain another proof for the unirationality of the bielliptic locus in the moduli space of curves of genus $g \geq 11$.
\end{abstract}

\maketitle

\section{Introduction}

The goal of this paper is to prove that the locus of bielliptic canonical curves inside the Hilbert scheme of genus $g \geq 11$ canonical curves in $\P^{g-1}$, \textit{i.e.} $\Hilb_{(2g-2)t + 1-g, \P^{g-1}}$, is unirational. As a consequence, we obtain another proof of the unirationality of the bielliptic locus of $M_g$, the coarse moduli space of genus $g$ curves.

A first description (at least to the author's knowledge) of the bielliptic locus in $M_g$ appears in the article \cite{art_cornalba_87}. One of its main results is a characterization of the irreducible components of the singular locus of $M_g$ (for $g \geq 4$) as loci of curves which are covers of prime degree, plus data related to the ramification. There, the bielliptic locus appears as one of these irreducible components and it has dimension $2g-2$. An improvement of the results in that paper appears in \cite{art_catanese_2010}. Later, in a series of papers by Bardelli, Casnati and Del Centina (\cite{art_biell_g3}, \cite{art_biell_g4}, \cite{art_biell_g5}), it is shown that the bielliptic loci are rational for small values of the genus, \textit{i.e.} $g \in \{3,4,5\}$. The methods employed in these articles are \textit{ad-hoc} and dependent on particular properties of the bielliptic curves of said genera. In the paper \cite{art_biell_g6}, the authors prove, among other things, that for $g \geq 6$, the bielliptic loci are unirational. Their methods are based on an analysis of the bielliptic curves among the tetragonal ones. One thing they could not prove in that paper is the rationality of the bielliptic locus when $g \geq 6$. This problem (at least to the author's knowledge) is still open.

In a more recent paper (\cite{art_aprodu_bruno_sernesi_2019}), the bielliptic canonical curves are characterized in terms of their second syzygy scheme (roughly the intersection of the quadrics generating their ideal, but which are involved in linear relations). From their results we can see that any bielliptic canonical curve of genus $g \geq 6$ is the intersection of a uniquely determined projective cone over a genus $1$ curve (we will refer to these as elliptic normal cones) with the unique quadric among the ones generating the ideal of the curve which is not involved in a linear relation. As a consequence of this description, we propose the following strategy for producing a family of bielliptic canonical curves which would dominate the locus of bielliptic canonical curves in $\Hilb_{(2g-2)t + 1-g, \P^{g-1}}$: first to obtain a parameter space for elliptic normal cones in $\P^{g-1}$, then to produce a projective bundle $\mathcal{P}$ over this space so that the fiber of this bundle over a point corresponding to a cone $\mathscr{C} \subseteq \P^{g-1}$ is isomorphic to the space of quadrics not containing $\mathscr{C}$, \textit{i.e.} $\Ho^0(\O_{\P^{g-1}}(2))/\Ho^0(\O_{\mathscr{C}}(2))$. After that, we can produce a map from $\mathcal{P}$ to $\Hilb_{(2g-2)t + 1-g, \P^{g-1}}$ which is dominant over the locus of bielliptic canonical curves. Using the universal family associated to this Hilbert scheme, we can map further to the bielliptic locus in $M_g$. By showing that $\mathcal{P}$ is unirational, we obtain a proof of the result mentioned in the first paragraph.

The paper is divided in three sections. The first one contains general properties of bielliptic curves. The second one is dedicated to studying elliptic normal cones and their properties. Its main point is to show that an appropriate parameter space for them is a locus inside a particular Hilbert scheme of surfaces and that said locus is unirational. Most of the results in this section are rather computational and not difficult to prove, but we include them for the completeness of the exposition. The last part is dedicated to proving the main result.

\textbf{Acknowledgements.} The author would like to thank his advisor, Marian Aprodu, for suggesting the problem, helpful conversations and comments throughout the preparation of this paper.

\begin{small}
The author was supported by the PNRR grant CF 44/14.11.2022 \textit{Cohomological Hall algebras of smooth surfaces and applications}.
\end{small}

\section{Preliminaries}

Throughout this paper we will work over the field of complex numbers.

\begin{defn}
    A nonsingular projective curve $X$ of genus $g \geq 2$ is called bielliptic if it admits a $2:1$ cover to a nonsingular genus $1$ curve, \textit{i.e.} there exist a genus $1$ curve $E$ and a morphism $\varphi: X \ra E$ of degree $2$.
\end{defn}

\begin{rem}
    The Riemann-Hurwitz theorem applied for $\varphi$ tells us that $K_X \sim R$ since the canonical bundle of a genus $1$ curve is trivial (we denote by $R$ the ramification divisor on $X$), and $\deg(R) = 2g - 2$. $R$ has to be in this case a sum of $2g - 2$ distinct points on $X$ because the cover has degree $2$, same for the branch locus $B = \varphi_{\ast}R$. So they are both reduced effective divisors on their respective curves.
\end{rem}

\begin{rem}\label{rem_split}
    A bielliptic curve $X$ will always have a non-trivial automorphism, namely the involution which exchanges the sheets of the cover. The existence of this involution leads to the splitting of the following exact sequence:
    \begin{align*}
        0 \ra \O_E \ra \varphi_{\ast}\O_X \ra L^{-1} \ra 0
    \end{align*}
where $L^{-1}$ is the cokernel of the natural morphism $\O_E \ra \varphi_{\ast}\O_X$, so $\varphi_{\ast}\O_X \simeq \O_E \oplus L^{-1}$ and $L^{\otimes 2} \simeq \O_E(B)$.
\end{rem}

\begin{rem}
There exist bielliptic curves of any genus $g \geq 2$. This is due to general methods of constructing double covers (see for example \cite[Ch. 2, pg. 46]{book_friedman_1998}) and the fact that if $D$ is a divisor on a nonsingular projective curve and $k \in \mathbb{N}$ divides its degree, then we can write $D = kD'$, for $D'$ another divisor.
\end{rem}

\begin{rem}
If $X$ is a bielliptic curve of genus $g$, then $X$ is not hyperelliptic for $g \geq 4$, $X$ is not trigonal for $g \geq 5$, and $X$ has a unique bielliptic structure for $g \geq 6$. These are consequences of the Castelnuovo-Severi inequality, which gives an upper bound for the genus of a curve covering two other curves. For a reference, see \cite[Ch. VIII, Exc. C-1, C-2]{book_acgh1}. Moreover, if $X$ is not trigonal, it is automatically tetragonal (compose the bielliptic map with a $2:1$ morphism from the genus $1$ curve to $\mathbb{P}^1$).
\end{rem}

From now on we will work with bielliptic curves of genus at least $6$, so we need not worry about multiple bielliptc structures on the same curve. Also, due to this restriction, we will always have canonical embeddings, so let us recall some basic properties of their canonical models. Suppose we have $C \subseteq \P^{g-1}$ a bielliptic canonical curve, with $\varphi: C \ra E$ as before. The general fiber of $\varphi$ has $2$ distinct points and we can consider the line in $\P^{g-1}$ which joins them.

\textbf{Claim:} Any $2$ such lines will intersect (this argument can be seen in \cite[Proposition 9]{art_aprodu_bruno_sernesi_2019}). 

Let $D \in \Div(C)$ be the sum of the $4$ points in two distinct general fibers. We have $D = \varphi^{\ast}D'$ for a divisor of degree $2$ on $E$. $\varphi$ is finite, so:
\begin{align*}
    \h^0(\varphi_{\ast}\O_C(D)) = \h^0(E,\varphi_{\ast}\O_C(D)) = \h^0(\O_E(D')) + \h^0(L^{-1} \otimes \O_E(D'))
\end{align*}
(we use the notation from \ref{rem_split}). By Riemann-Roch, $\h^0(\O_E(D')) = 2$ and $\h^0(L^{-1} \otimes \O_E(D')) = 0$. Therefore, we conclude that $\dim|D| = 1$, and by the geometric Riemann-Roch theorem (see \cite[Ch. I, pg. 12]{book_acgh1}) we have that $\dim(\overline{\Phi_K(D)}) = 2$, which means that our $4$ initial points lie, along with the $2$ lines, in the same $2$-dimensional subspace in $\P^{g-1}$, so the lines intersect. Let us remark that the same argument works if at least one of the fibers we consider would be over a point of the branch locus; then we would work with points with multiplicity $2$ on $C$ and with tangent lines instead of secants. So the previous claim is true for all lines generated by the points of the fibers, and by some simple projective geometry arguments, all these lines actually intersect in the same point $V$, which is uniquely determined by the bielliptic structure.

The union of these lines is an irreducible surface, namely a projective cone over $E$ embedded as a degree $g-1$ curve in a hyperplane of $\P^{g-1}$, and the projection from the vertex $V$ onto said hyperplane induces the double cover when restricted to $C$. This embedding of $E$ is called an elliptic normal curve and the cone, which we will denote by $\mathscr{C}$, is called an elliptic normal cone.

Next, since $C$ is tetragonal, we have that $\cliff(C) = 2$, so by Petri's theorem the ideal of $C$ is generated by quadrics, and the relations between those quadrics are not generated by the linear ones (that would happen if $\cliff(C) \geq 3$). From \cite{art_aprodu_bruno_sernesi_2019} we know that:\\
i) $\kappa_{11}(\mathscr{C}) = \kappa_{11}(C) - 1$, so the space of quadrics containing the cone has codimension $1$ in the space of quadrics containing $C$;\\
ii) $\syz_2(C) = \mathscr{C}$, so the cone is the intersection of the quadrics containing $C$ which are involved in linear relations.\\
We gather what we discussed so far in the following proposition:

\begin{prop}
    Let $C \subseteq \P^{g-1}$ be a canonical bielliptic curve of genus at least $6$, with bielliptic structure given by $\varphi: C \ra E$. Then:\\ a) there exists a unique point $V \in \P^{g-1}$ which is the vertex of a unique cone $\mathscr{C}$ over an embedding of $E$ in some hyperplane - both of which are determined by $\varphi$;\\
    b) the projection of $C$ onto the embedding of $E$ from the vertex $V$ is (modulo isomorphisms) exactly the map $\varphi$;\\
    c) $C$ is the intersection of $\mathscr{C}$ with the unique quadric which is not involved in a linear relation.
\end{prop}

The conclusion is the following: a better understanding of elliptic normal cones is a first step towards the study of bielliptic curves.

\section{Cones over elliptic normal curves}

This section is dedicated to the study of elliptic normal cones in $\P^n$. The main purpose is to exhibit a parameter space for all such cones lying in $\P^n$.

\subsection{Hilbert polynomial of elliptic cones}

We will need the following well-known result:

\begin{prop}\label{prop_proj_norm}
    A projective cone in $\P^n$ over a projectively normal variety contained in a hyperplane is projectively normal (therefore normal).
\end{prop}

The proof follows immediately from the description of homogeneous coordinate rings and \cite[Ch. 5, Exc. 9]{book_am_69}. In particular cones over elliptic normal curves (which are projectively normal, see \cite[Proposition IV.1.2]{hulek_ell}) are projectively normal.
\begin{prop}
Let $\mathscr{C} \subseteq \P^n$ be an elliptic normal cone over $E$. Then its Hilbert polynomial is $Q_n = \frac{nt^2}{2} + \frac{nt}{2}$.
\end{prop}
\begin{proof}
Without loss of generality we can assume the vertex of $\mathscr{C}$ is $[0:\ldots:0:1]$. If $S = k[X_0, \ldots, X_n]$ and $I$ is the ideal of $\mathscr{C}$, then $(S/I)/X_n(S/I)$ is the homogeneous coordinate ring of $E$. Consider now the following exact sequence of graded modules:
\begin{align*}
    0 \ra (S/I)(-1) \ra S/I \ra (S/I)/X_n(S/I) \ra 0
\end{align*}
From this we obtain the following relation between Hilbert functions:
\begin{align*}
    H(S/I,t) - H(S/I,t-1) = H((S/I)/X_n(S/I),t)
\end{align*}
$E$ is an eliptic normal curve, so $H((S/I)/X_n(S/I),t) = nt$. Therefore:
\begin{align*}
    H(S/I,t) - H(S/I,t-1) = nt
\end{align*}
and from this recurrence relation we obtain:
\begin{align*}
\begin{split}
    H(S/I,t) &= \frac{nt^2}{2} + \frac{nt}{2} + H(S/I,0)\\
             &= \frac{nt^2}{2} + \frac{nt}{2} + \chi(\O_{\mathscr{C}})
\end{split}
\end{align*}
We are left with computing $\chi(\O_{\mathscr{C}})$. Let $f: X \ra \mathscr{C}$ be the blowup of the cone's vertex. Then $X$ is a $\P^1$-bundle over $E$, so $X = \mathbb{P}(\mathcal{E})$, where $\mathcal{E}$ is a rank $2$ bundle over $E$. In particular, it is a ruled surface over $E$. We use the Leray spectral sequence:
    \begin{align*}
        E_2^{pq} = \Ho^q (\mathscr{C}, R^p f_{\ast} \O_X ) \Rightarrow \Ho^{p+q}(X,\O_X)
    \end{align*}
    We know that $R^p f_{\ast} \O_X = 0$ for all $p \geq 2$ since the dimension of the fibers of $f$ is at most $1$ (\cite[III.11.2]{book_hartshorne_1977}). All the differentials on page $2$ are $0$, so $E_{\infty}^{pq} = E_2^{pq}$ for all $p, q$, therefore $\Ho^i(\mathscr{C},f_{\ast}\O_X) \simeq \Ho^i(X,\O_X)$ for all $i$. Since $\mathscr{C}$ is normal, we have $f_{\ast}\O_X = \O_{\mathscr{C}}$, and by the previous isomorphisms:
    \begin{align*}
        \h^0(\mathscr{C},\O_{\mathscr{C}}) = 1\\
        \h^1(\mathscr{C},\O_{\mathscr{C}}) = 1\\
        \h^2(\mathscr{C},\O_{\mathscr{C}}) = 0
    \end{align*}
\end{proof}

We can thus conclude that a suitable parameter space for such cones is a locus inside the Hilbert scheme $\Hilb_{Q_n,\P^n}$ of subschemes in $\P^n$ having Hilbert polynomial $Q_n$. For a succinct introduction to the Hilbert scheme and its basic properties, see the first chapter of \cite{book_hm_98}. The points of this scheme parametrize (among other subschemes) irreducible, non-degenerate surfaces of degree $n$ in $\P^n$. We know from \cite[11. Theorem 8]{nagata_rat_1} that any such surface has to be one of the following (provided $n \geq 10$): a projection of an irreducible, non-degenerate surface of degree $n + 1$ in $\P^{n+1}$ with center outside the surface or an elliptic normal cone. We will next show that the projections do not have $Q_n$ as their Hilbert polynomial.

Suppose $X \in \P^n$ is an irreducible, non-degenerate surface of degree $n$ obtained by projecting an irreducible, non-degenerate surface $S$ of degree $n$ in $\P^{n+1}$ from a point $O \not\in S$. By the classification of surfaces of minimal degree (see \cite[page 525]{book_gh_94}), $S$ is a rational normal scroll, either an embedding of a Hirzebruch surface, or a cone over a rational normal curve. Either way, its general hyperplane section is a rational normal curve and we can consider that the hyperplane does not contain the center $O$. This way, no line joining $O$ with a point on this rational normal curve can be a tangent or a bisecant. We only care to eliminate these because rational normal scrolls have the following property: if a line intersects $S$ in at least $3$ points, then it is contained in $S$, so a line through $O$ which intersects $S$ can cut it in at most one other point. The projection of this hyperplane section of $S$ will be a hyperplane section of $X$, nonsingular and rational (although not a rational normal curve because of the degree). Then, by a computation as in the case of elliptic normal cones, the Hilbert polynomial of $X$ is:
\begin{align*}
    P_X(t) = \frac{n}{2}t^2 + \frac{n+2}{2}t + P_X(0)
\end{align*}
It is clear that these projections of scrolls will not correspond to points of $\Hilb_{Q_n,\P^n}$.

In conclusion, out of all the irreducible, non-degenerate surfaces in $\P^n$, where $n \geq 10$ (this restriction will be imposed throughout the rest of the paper), the only ones which are parametrized by points in $\Hilb_{Q_n, \P^n}$ are the elliptic normal cones. All the other points will correspond to closed subschemes which are not as well-behaved (reducible, maybe with embedded points, with components of different dimensions, etc). But we can resort to working with the open subscheme of irreducible, non-degenerate surfaces in $\Hilb_{Q_n, \P^n}$, and this will be our parameter space for elliptic normal cones in $\P^n$. We will from now on denote this space by $H_c$, taken with its reduced structure just as a safety measure.

\subsection{Elliptic cones and ruled surfaces}

Let $\mathscr{C} \subseteq \P^n$ be an elliptic normal cone over the genus $1$ curve $E$. As mentioned before, by blowing up the vertex of $\mathscr{C}$ we obtain a ruled surface $X$ over $E$. Moreover, we have $X \simeq \P(\O_E \oplus \O_E(-D))$ (see \cite[V.2.11.4.]{book_hartshorne_1977}), with $\O_E(D) = \O_E(1)$, so $\deg(D) = n $. The rank $2$ bundle inducing the ruled surface is normalized, so the invariant of the surface is $e = - \deg(\O_E(-D)) = n$. Denoting the unique minimal section by $C_0$, we have that $C_0^2 = -n$ and it is the inverse image of the vertex of the cone. We also have $\Pic(X) \simeq \mathbb{Z}C_0 \oplus \pi^{\ast}\Pic(E)$ (see \cite[V.2.3.]{book_hartshorne_1977}).  

\begin{prop}\label{prop_ruled_lin_sys}
Let $\pi: X_D = \P(\O_E \oplus \O_E(-D)) \ra E$ be a ruled surface over the genus $1$ curve $E$, where $D$ is a degree $n \geq 3$ divisor on $E$. Then there exists a unique divisor $H_D$ on $X_D$ such that:\\
i) the general member of the complete linear system $|H_D|$ is an irreducible, nonsingular curve of genus $1$;\\
ii) $|H_D|$ is base-point-free;\\
iii) the morphism $\varphi_{|H_D|}$ induced by $|H_D|$ corresponds to blowing up the vertex of an elliptic normal cone in $\P^n$.
\end{prop}
\begin{proof}
For this proof we will denote the surface and the divisor simply by $X$ and $H$, respectively. Let us first remark that any such divisor must satisfy the following conditions: $H.C_0 = 0$, $H.f = 1$, $H^2 = n$ and $\h^0(\O_X(H)) = n + 1$. We will show that they uniquely determine $H$.

We have $H = aC_0 + \beta f$, $a \in \mathbb{Z}$, $\delta = \deg_E(\beta)$. If $H.f = 1$ and $H.C_0 = 0$, then $a = 1$ and $\delta = n$. Therefore $H = C_0 + \beta f$. Next, we'll show that $\beta = D$. We have that:
\begin{align*}
    \begin{split}
        \pi_{\ast}\O_X(C_0 + \beta f) &= \pi_{\ast} (\O_X(C_0) \otimes \pi^{\ast}\O_E(\beta))\\
        &= \pi_{\ast}(\O_X(C_0)) \otimes \O_E(\beta)\\
        &= (\O_E \oplus \O_E(-D)) \otimes \O_E(\beta)\\
        &= \O_E(\beta) \oplus \O_E(\beta - D)
    \end{split}
\end{align*}
Since $H.f \geq 0$, $\h^i(\O_X(H)) = \h^i(\pi_{\ast}\O_X(H))$ for all $i$ (\cite[V.2.4]{book_hartshorne_1977}), so $\h^0(\O_X(H)) = \h^0(\O_E(\beta)) + \h^0(\O_E(\beta - D))$. $\beta$ and $D$ are both divisors of degree $n$ on $E$, and since $\h^0(\O_X(H)) = n + 1$, then $\h^0(\O_E(\beta - D)) = 1$. But $\beta - D$ is of degree $0$, so $\beta = D$. Therefore the only possibility is $H = C_0 + Df$. Next we will prove that $H$ satisfies the three properties. It is easy to see that $iii)$ is a consequence of the first two, so we only need to show that it satisfies $i)$ and $ii)$.

$i)$ Since $\h^0(\O_X(H)) = n + 1$, but $\h^0(\O_E(D)) = n$, the part of $|H|$ consisting of divisors of type $C_0 + D'f$ with $D'$ linearly equivalent to $D$ has codimension $1$, thus the general member cannot be of this form. From now on, when we refer to a general member of this linear system it means that it is not of the form $C_0 + D'f$, with $D' \sim D$ in $\Pic(E)$.

Notice that if $H' \in |H|$ is a general member, then $C_0 \not\subseteq \Supp(H')$. Otherwise $H' - C_0 \sim Df$, so $H'$ would be of the form $C_0 + D'f$, with $D' \sim D$ in $\Div(E)$. 

Let $C$ be an irreducible curve in the support of $H'$. Then $C = aC_0 + \beta f$, with $\deg(\beta) = b$. As mentioned before, $C \neq C_0$. Suppose that $C \neq f$ (clearly $H'$ cannot have only fibers in its support). Then $a>0$ and $b \geq an$ (see \cite[V.2.20]{book_hartshorne_1977}). We also have that $|H' - C| \neq \varnothing$. By Serre duality:
    \begin{align*}
        \begin{split}
            \h^0(\O_X(H' - C)) &= \h^2(\O_X(K_X - H' + C))\\
                               &= \h^2(\O_X((a-3)C_0 + (\beta - 2D)f))
        \end{split}
    \end{align*}
    By \cite[V.2.4]{book_hartshorne_1977}, the above quantity is $0$ if $((a-3)C_0 + (\beta - 2D)f.f) = a - 3 \geq 0$, so we must have $a \in \{1,2\}$. But since $H' \equiv C_0 + nf$, we cannot have curves $C$ in the support of $H'$ with $a = 2$. So $a = 1$ and $b \geq n$. $|H' - C| = |(D-\beta)f| \neq \varnothing$ and $((D-\beta)f.f) = 0$, so by \cite[V.2.1,V.2.4]{book_hartshorne_1977}, $\pi_{\ast}(\O_X((D-\beta)f))$ is a line bundle of degree $n - b \leq 0$ on $E$ and has the same cohomology as $\O_X((D-\beta)f)$, in particular it has at least a non-zero global section. As a consequence, its degree must be $0$, so it is $\O_E$, \textit{i.e.} $D \sim \beta$, and $H' = C$, \textit{i.e.} $H'$ is an irreducible curve.

Let $C$ be a general member of $|H|$. By the adjunction formula:
\begin{align*}
	2p_a(C) - 2 = (C.C + K_X) = 0
\end{align*}
so $p_a(C) = 1$. Let $\Tilde{C}$ be the normalization of $C$. It is an integral, nonsingular curve. Then:
    \begin{align*}
        1 = p_a(C) = p_a(\Tilde{C}) + \sum\limits_{P \in \Sing(C)} \delta_P
    \end{align*}
    so either $p_a(\Tilde{C}) = 0$ and $C$ has only one singular point, or $p_a(\Tilde{C}) = 1$ and $C$ is nonsingular. In the first case $\Tilde{C} \simeq \P^1$, so by composing the normalization map with $\pi$ we get a finite nonconstant morphism from $\P^1$ to $E$, which is a contradiction. So we can only have $C$ nonsingular of genus $1$.

$ii)$ Since $H^2 = n$, any $2$ general members of the linear system intersect in at most $n$ distinct points, so there can be only finitely many base points. $H.C_0 = 0$ implies that the base points can not lie on $C_0$, so in $\Supp(H)$ the base points would have to be on the fibers in $Df$. Also, $H.f = 1$ implies that there can be at most one base point per fiber. In this case, if $|H|$ were not base-point-free, $|Df|$ would not be either, since through a point passes an unique fiber, so $\dim|Df| = 0$, which would be a contradiction.
\end{proof}

Thus we can see that any elliptic normal cone in $\P^n$ can be obtained as the image of a morphism from a ruled surface $\P(\O_E \oplus\O_E(-D))$ over a genus $1$ curve $E$ given by the linear system $|H_D|$, where $D$ is a degree $n$ divisor on $E$ and $H_D = C_0 + Df$. We can further refine this description due to the following lemma.

\begin{lemma}
    Let $E$ be a genus $1$ curve and $L$ a degree $n$ line bundle on $E$. Then there exists $P \in E$ such that $L \simeq \O_E(nP)$.
\end{lemma}
\begin{proof}
    We have a map $\Pic_1(E) \ra \Pic_n(E)$ sending $\O_E(P)$ to $\O_E(nP)$. This is a nonconstant map between two curves, therefore it is surjective.
\end{proof}

So in our previous description, it is sufficient to work with divisors $nP$, where $P \in E$, instead of all degree $n$ divisors. One more thing to mention is that a linear system of type $|H_{nP}|$ gives a morphism from $\P(\O_E \oplus\O_E(-nP))$ to $\P(\Ho^0(\O(H_{nP}))^{\vee})$. Therefore, to obtain a morphism from the ruled surface to $\P^n$ we just need to compose the previously mentioned one with an isomorphism $\P(\Ho^0(\O(H_{nP}))^{\vee}) \ra \P^n$.

\begin{prop}
Let $E$ be a genus $1$ curve. Then any two elliptic normal cones in $\P^n$ obtained by the process described above are projectively equivalent.
\end{prop}
\begin{proof}
Let $P, Q \in E$ be two distinct points. Then we have the ruled surfaces  $X_{nP} = \P(\O_E \oplus\O_E(-nP))$ and  $X_{nQ} = \P(\O_E \oplus\O_E(-nQ))$ over $E$ and the following diagram:
\begin{align*}
    \begin{tikzcd}[ampersand replacement=\&, column sep = large, row sep = small]
        X_{nP} \arrow[r,"\varphi_{|nP|}"] \& \P(\Ho^0(\O_{X_{nP}}(H_{nP}))^{\vee}) \arrow[swap,
       start anchor={[shift={(-3pt,9pt)}]south east},
       end anchor={[shift={(2pt,-4pt)}]north west}, 
       dr,"\psi_1"']\\       
        {} \& {} \& \P^n\\
        X_{nQ} \arrow[r,"\varphi_{|nP|}"] \& \P(\Ho^0(\O_{X_{nQ}}(H_{nQ}))^{\vee})  \arrow[swap,
       start anchor={[shift={(-2pt,10pt)}]south east},
       end anchor={[shift={(2pt,-13pt)}]north west},
       ur,"\psi_2"]
   \end{tikzcd}
\end{align*}
with $\psi_1$, $\psi_2$ isomorphisms.

Let $f: E \ra E$ be the automorphism given by the translation which sends $P$ to $Q$. $f$ induces an isomorphism between $\Ho^0(\O_E(nP))$ and $\Ho^0(\O_E(nQ))$. As in the proof of \ref{prop_ruled_lin_sys} we have
\begin{align*}
 \Ho^0(\O_{X_{nP}}(H_{nP})) \simeq \Ho^0(\O_E) \oplus \Ho^0(\O_E(nP))
\end{align*}
and
\begin{align*}
\Ho^0(\O_{X_{nQ}}(H_{nQ})) \simeq \Ho^0(\O_E) \oplus \Ho^0(\O_E(nQ))
\end{align*}
As a consequence we get an isomorphism
\begin{align*}
\Tilde{f}: \P(\Ho^0(\O_{X_{nP}}(H_{nP}))^{\vee}) \ra \P(\Ho^0(\O_{X_{nQ}}(H_{nQ}))^{\vee})
\end{align*}
and $\psi_2 \circ \Tilde{f} \circ \psi_1^{-1}$ gives the projective equivalence in $\P^n$ between the images of $\psi_1 \circ \varphi_{|nP|}$ and $\psi_2 \circ \varphi_{|nQ|}$.
\end{proof}

We conclude this section with the following result:

\begin{thm}
The reduced locus $H_c$ of elliptic normal cones in the Hilbert scheme $\Hilb_{Q_n,\P^n}$ is irreducible and unirational.
\end{thm}
\begin{proof}
We will start by constructing a space which parametrizes tuples of type $(E,P,s_1, \ldots , s_n)$, where $(E,P)$ is an elliptic curve and $s_1, \ldots , s_n$ is a basis of $\Ho^0(E, \O_E(nP))$. We have the following smooth, projective family of genus $1$ curves:\\
\begin{align*}
    \begin{tikzcd}[ampersand replacement=\&]
        \mathcal{C} = \{ zy^2 = x(x - z)(x - \lambda z) \} \arrow[d,"\pi"'] \arrow[hookrightarrow]{r} \& \P^2 \times A \\
        A \& {}\\
   \end{tikzcd}
\end{align*}
where $A = \mathbb{A}^1_{\lambda} \setminus \{0,1\}$. $\pi$ admits a section $P$ which sends $\lambda$ to $([0:1:0], \lambda )$. Due to this we can define a line bundle on $\mathcal{C}$, $\O_{\mathcal{C}}(nP)$, whose restriction to a fiber of $\pi$ over $\lambda_0$ is $\O_{\mathcal{C}_{\lambda_0}}(nP( \lambda_0 ))$, a degree $n$ line bundle on the genus $1$ curve $\mathcal{C}_{\lambda_0}$. By applying Grauert's theorem (see \cite[III.12.9]{book_hartshorne_1977}), we obtain the rank $n$ bundle $\mathcal{V}_n = \pi_{\ast} \O_{\mathcal{C}}(nP)$ on $A$. We then take the global $\Spec$ of the sheaf of $\O_A$-algebras $\Sym ( \Hom( \O_A^{\oplus n} , \mathcal{V}_n )^{\vee}) $, \textit{i.e.} the geometric vector bundle associated to $\Hom( \O_A^{\oplus n} , \mathcal{V}_n )$. This scheme parametrizes tuples of type $(E,P,s_1, \ldots , s_n)$, where $s_1, \ldots , s_n$ are elements of $\Ho^0(\O_E(nP))$. The one we mentioned in the beginning, where the $s_i$'s form a basis, is an open subscheme which we will denote by $\Sigma$. Since $A$ is rational, it is clear that $\Sigma$ is rational as well.

Before going further, let us simplify the notation by writing $s$ instead of $s_1, \ldots , s_n$ for a basis of $\Ho^0(\O_E(nP))$. Moreover, note that to any $s$ we can associate an isomorphism $\psi_s : \P(\Ho^0(\O_{X_{nP}}(H_{nP}))^{\vee}) \ra \P^n$ (of course, this correspondence is only surjective).

Let $\Xi$ be the following incidence scheme taken with its reduced structure:
\begin{align*}
    \Xi = \{ (E,P,s,z) \in \Sigma \times \P^n | z \in \psi_s(\varphi_{|H_{nP}|}(X_{nP})) \} \subseteq \Sigma \times \P^n
\end{align*}
To make the notation clearer, to any tuple $(E,P,s)$ we can associate the ruled surface $X_{nP}$ over $E$, the morphism $\varphi_{|nP|}$ to $\P(\Ho^0(\O_{X_{nP}}(H_{nP}))^{\vee})$ whose image is an elliptic normal cone and the isomorphism of projective spaces $\psi_s$. We compose $\varphi_{|nP|}$ with $\psi_s$ in order to have all the cones that we obtain in the same projective space. In $\Xi$, the points $z$ belong to these elliptic normal cones in $\P^n$. We use all the isomorphisms $\psi_s$ because, due to the previous proposition, we obtain in this way all the possible elliptic normal cones of $\P^n$.

Now consider the projection $\Xi \ra \Sigma$. By generic flatness (see \cite[Theorem 6.9.1, pg. 153]{book_ega_4pt2}), we can work over an open dense subset $U$ of $\Sigma$ in order for the projection $\Xi_U \ra U$ to be flat. \textit{A priori} there is no reason for $\Xi_U \ra U$ to have all the scheme-theoretic fibers reduced, but over an open subset $U'$ of $U$ they are (see \cite[Theorem 12.2.4, pg. 183]{book_ega_4pt3}) and the generic fiber is reduced since $\Xi$ is a variety. This gives us a family of elliptic normal cones in $\P^n$ over an open subset of $\Sigma$. Therefore, we obtain a dominant rational map $\Sigma \dashedrightarrow H_c$. As $\Sigma$ is irreducible and rational, we get that $H_c$ is irreducible and unirational.
\end{proof}

\begin{rem}
The construction of $\Xi$ can also be used to prove the unirationality of the open subscheme of $\Hilb_{nt,\P^{n-1}}$ which parametrizes elliptic normal curves in $\P^{n-1}$ (let us denote this open subscheme by $H_e$). Namely, if we would consider just like in the previous proof:
\begin{align*}
    \Xi' = \{ (E,P,s,z) \in \Sigma \times \P^{n-1} | z \in \psi_s(\varphi_{|nP|}(E)) \} \subseteq \Sigma \times \P^{n-1}
\end{align*}
where this time $\varphi_{|nP|}$ embeds $E$ in $\P(\Ho^0(\O_E(nP))^{\vee})$ and $\psi_s$ is an isomorphism between  $\P(\Ho^0(\O_{E}(nP))^{\vee})$ and $\P^{n-1}$, then by the same arguments we can obtain a dominant rational map $\Sigma \dashedrightarrow H_e$.
\end{rem}

\section{Main result}

As mentioned before, the space of quadrics containing a bielliptic canonical curve of genus $g$ is determined modulo the ones vanishing on the unique elliptic normal cone which contains it. What we need next is a projective bundle over $H_c$ whose fiber over a point $[\mathscr{C}]$ is isomorphic to the the projectivisation of the space $\Ho^0(\O_{\P^n}(2))/\Ho^0(\I_{\mathscr{C}}(2))$ (this vector space is isomorphic to $\Ho^0(\O_{\mathscr{C}}(2))$ due to the projective normality of the cones). The following proposition shows that the dimension of this space is constant for any elliptic normal cone $\mathscr{C}$ in $\P^n$.

\begin{prop}
Let $\mathscr{C} \subseteq \P^n$ be an elliptic normal cone over $E$. Then \\$\h^0(\O_{\mathscr{C}}(2)) = 3n + 1$.
\end{prop}
\begin{proof}
Notice that the vector spaces $\Ho^0(\P^n,\I_{\mathscr{C}}(2))$ and $\Ho^0(\P^{n-1},\I_{E}(2))$ have the same dimension, so we can do the computations for $E$. We start with the following exact sequence:
\begin{align*}
    0 \ra \I_{E}(2) \ra \O_{\P^{n-1}}(2) \ra \O_{E}(2) \ra 0
\end{align*}
$\h^1(\I_{E}(2)) = 0$ by the projective normality of $E$. $E$ is embedded via a line bundle of degree $n$, so by Riemann-Roch $\h^0(\O_{E}(2)) = 2n$. Therefore:
\begin{align*}
    \h^0(\I_{\mathscr{C}}(2)) = \h^0(\O_{\P^{n-1}}(2)) - \h^0(\O_{E}(2)) = \frac{n(n+1)}{2} - 2n = \frac{n^2 - 3n}{2}
\end{align*}
and:
\begin{align*}
    \begin{split}
        \h^0(\O_{\mathscr{C}}(2)) &= \h^0(\O_{\P^n}(2)) - \h^0(\I_{\mathscr{C}}(2))\\
        &= \binom{n+2}{2} - \frac{n^2 - 3n}{2} = 3n + 1
    \end{split}
\end{align*}
\end{proof}

In the case $n = g - 1$, so $g \geq 11$, for a fixed elliptic normal cone $\mathscr{C} \subseteq \P^{g-1}$ the space of quadrics modulo the ones vanishing on $\mathscr{C}$ is $\P(\Ho^0(\mathscr{C},\O_{\mathscr{C}}(2))$, a projective space of dimension $3g-2$. Since the bielliptic canoncal curves contained in $\mathscr{C}$ are intersections of $\mathscr{C}$ with a quardic, a parameter space for them would be a projective bundle on $H_c$, whose fiber over a point $[\mathscr{C}]$ would be $\P(\Ho^0(\mathscr{C},\O_{\mathscr{C}}(2))$. This is the strategy we will employ in order to prove the main result.
\begin{thm}
	The bielliptic locus in the Hilbert scheme of canonical curves of genus $g \geq 11$ is unirational. Furthermore, the bielliptic locus in $M_g$ is unirational for $g \geq 11$.
\end{thm}
\begin{proof}
Let $q: \mathcal{C} \ra H_c$ be the universal family of elliptic normal cones in $\P^{g-1}$. We have the inclusion $i: \mathcal{C} \ra \P^{g-1} \times H_c$ and $p: \mathcal{C} \ra \P^{g-1}$ the composition of the first projection and $i$ (see diagram below).
\begin{align*}
    \begin{tikzcd}[ampersand replacement=\&]
        \mathcal{C} \arrow[d,"q"'] \arrow[r,"i"] \arrow[bend left, rr,"p"] \& \P^{g-1} \times H_c \arrow[dl] \arrow[r] \& \P^{g-1} \\
        H_c \& {} \& {} \\
   \end{tikzcd}
\end{align*}
Consider the bundle $p^{\ast}\O(2)$ on $\mathcal{C}$ which restricts over a fiber of $q$ (which is an elliptic normal cone $\mathscr{C}$ in $\P^{g-1}$) to $\O_{\mathscr{C}}(2)$, then take $q_{\ast}p^{\ast}\O(2)$. The base of the family is integral and the cohomology on the fibers of $q$, \textit{i.e.} $\Ho^0(\mathscr{C},\O_{\mathscr{C}}(2))$, has constant dimension $3g-2$ from the previous proposition, therefore by Grauert's theorem $q_{\ast}p^{\ast}\O(2)$ is locally free of rank $3g-2$ and for all points $[\mathscr{C}] \in H_c$ the fiber of $q_{\ast}p^{\ast}\O(2)$ over $[\mathscr{C}]$ is isomorphic to $\Ho^0(\mathscr{C},\O_{\mathscr{C}}(2))$, which is exactly what we needed. Now take its projectivisation $\mathcal{P} = \P(q_{\ast}p^{\ast}\O(2))$ and this is the projective bundle we were looking for. It is a projective bundle over $H_c$ whose points parametrize pairs $(\mathscr{C},Q)$, where $\mathscr{C} \subseteq \P^{g-1}$ is an elliptic normal cone and $Q \subseteq \P^{g-1}$ is a quadric hypersurface not containing $\mathscr{C}$. Consider the incidence scheme (again with the reduced structure):
\begin{align*}
\Phi = \{ (\mathscr{C},Q,z) | z \in \mathscr{C} \cap Q \} \subseteq \mathcal{P} \times \P^{g-1}
\end{align*}
and the first projection $\Phi \ra \mathcal{P}$. We can reduce again to an open dense subset of $\mathcal{P}$ so that we obtain a flat family, then again to have reduced, nonsingular fibers. We thus obtain a flat family of bielliptic canonical curves in $\P^{g-1}$, which induces a rational map $\mathcal{P} \dashedrightarrow \Hilb_{(2g-2)t + 1 - g, \P^{g-1}}$ dominating the locus of bielliptic curves in the Hilbert scheme of canonical curves.

Since $H_c$ is unirational, $\mathcal{P}$ is unirational as well, and because of the map $\mathcal{P} \dashedrightarrow \Hilb_{(2g-2)t + 1 - g, \P^{g-1}}$ we can conclude that the locus of bielliptic curves in the Hilbert scheme of canonical curves is unirational. Now the unirationality of the bielliptic locus in $M_g$ is automatic because the previously mentioned locus maps to the bielliptic locus in $M_g$.
\end{proof}

\end{document}